\documentclass[12pt]{amsart}
\usepackage{amsmath,amscd,amssymb,amsfonts}
\usepackage[all]{xy}

\newcommand{\ver}{September 25, 2008}
\newcommand{\bC}{{\mathbb C}}

\newcommand{\bN}{{\mathbb N}}
\newcommand{\bP}{{\mathbb P}}
\newcommand{\bR}{{\mathbb R}}
\newcommand{\bQ}{{\mathbb Q}}
\newcommand{\bZ}{{\mathbb Z}}
\newcommand{\bx}{{\bf x}}
\newcommand{\cA}{{\mathcal A}}

\newcommand{\cE}{{\mathcal E}}
\newcommand{\cF}{{\mathcal F}}
\newcommand{\cG}{{\mathcal G}}
\newcommand{\cH}{{\mathcal H}}
\newcommand{\cL}{{\mathcal L}}

\newcommand{\cJ}{{\mathcal J}}
\newcommand{\cO}{{\mathcal O}}

\newcommand{\cV}{{\mathcal V}}

\newcommand{\Hom}{{\rm{Hom}}}
\newcommand{\Pic}{{\rm{Pic}}}

\newcommand{\wti}{\widetilde}

\newcommand{\Gr}{\text{\rm Gr}}

\newcommand{\ra}{\rightarrow}
\newcommand{\rndown}[1] {\llcorner {#1} \lrcorner}

\theoremstyle{plain}
\newtheorem{thm}{Theorem}[section]
\newtheorem{cor}[thm]{Corollary}

\newtheorem{lem}[thm]{Lemma}
\newtheorem{prop}[thm]{Proposition}

\theoremstyle{definition}

\newtheorem{rem}[thm]{Remark}

\title[Spectrum of hyperplane arrangements]{Hodge spectrum of hyperplane arrangements}
\author{Nero Budur}
\address{Department of Mathematics, The
University of Notre Dame, IN 46556, USA} \email{nbudur@nd.edu}

\date{\ver}

\thanks {The author was supported by the NSF grant DMS-0700360.}

\begin{document}

\maketitle

\begin{abstract}
In this article there are two main results. The first result gives
a formula, in terms of a log resolution, for the graded pieces of
the Hodge filtration on the cohomology of a unitary local system
of rank one on the complement of an arbitrary divisor in a smooth
projective complex variety. The second result is an application of
the first. We give a combinatorial formula for the spectrum of a
hyperplane arrangement. M. Saito recently proved that the spectrum
of a hyperplane arrangement depends only on combinatorics.
However, a combinatorial formula was missing. The formula is
achieved by a different method.
\end{abstract}

\section{Introduction}

In this article there are two main results. The first result,
Theorem \ref{mainThmsectiononLocSys}, is concerned with the
computation of the Hodge filtration on the cohomology of local
systems on the complement $U$ of an arbitrary divisor $D$ in a smooth
complex projective variety $X$. For a unitary local system of rank
one $\cV$ on $U$, we give a formula in terms of a log resolution
of $(X,D)$ for the graded pieces $\Gr_F^pH^m(U,\cV)$. This formula
is related to the multiplier ideals of $(X,D)$ and generalizes
\cite{Bu-ULS} - Proposition 6.4  and part of \cite{DS}- Theorem 2. It is also related to \cite{Bu-ULS}- Theorems 1.3, 1.4, and \cite{L07}- Theorem 2.1. Hodge numbers of  local systems on the complements of planar divisors and of isolated non normal crossings divisors, in both local and global case, were discussed in \cite{L83}, \cite{L01}, \cite{L03}, \cite{L04}.

\medskip

The second result is an application of Theorem
\ref{mainThmsectiononLocSys} and concerns the Hodge spectrum of a
hyperplane arrangement. For any closed subscheme $D$ of $X$, the
spectrum, as the multiplier ideals and the b-function, is a
measure of the complexity of the singularities of $D$. When $D$ is
a hypersurface with an isolated singularity, the spectrum enjoys a
semicontinuity property which has been very useful for the
deformation theory of such singularities (see \cite{Ku} and
references there). Less is known about the spectrum for arbitrary
singularities. All three types of invariants (spectrum, multiplier
ideals, b-functions) are notoriously difficult to compute, see
\cite{MS-Survey}. Despite implementations of algorithms in
programs like Macaulay 2, Singular, and Risa/Asir, computation in
the cases when the dimension of $X$ is $\ge 3$ is very expensive.
For the class of varieties defined by monomial ideals it can be
said that there are satisfying formulas for all three notions in
terms of combinatorics, making the computations faster
(\cite{Ho}), \cite{DMS}, \cite{BMS1}, \cite{BMS2}). For the class
of hyperplane arrangements, some invariants (such as the ring
$H^*(X-D,\bZ)$, see \cite{OS}) turned out to depend only on
combinatorics. Hence it is natural to ask if the information from
multiplier ideals, spectra, and b-functions for this class is
combinatorially determined. M. Musta\c{t}\u{a} \cite{Mu} (see also
\cite{Te}) gave a formula for multiplier ideals of a hyperplane
arrangement. However, it was not clear that the jumping numbers
(the most basic numerical invariants that come out of the
multiplier ideals) admit a combinatorial formula. M. Saito
\cite{MS} then proved that the spectrum and the jumping numbers of
a hyperplane arrangement depend only on combinatorics. However, a
combinatorial formula was missing. A different proof and a
combinatorial formula for the jumping numbers and for the
beginning piece of the spectrum was given in \cite{Bu-JNHA}. In
this article we give a combinatorial formula and a different proof
of the combinatorial invariance for the spectrum of a hyperplane
arrangement, see Theorem \ref{computationSpectrumFinal}. The
b-function for hyperplane arrangements remains yet to be
determined. See \cite{MS-b} for the latest advances.

\medskip

The structure of the article is the following. In section 2 we fix
notation and review the multiplier ideals, Hodge spectrum, and
intersection theory. In section 3 we prove Theorem
\ref{mainThmsectiononLocSys} on the Hodge filtration for local
systems, based on the geometrical interpretation of rank one
unitary local systems from \cite{Bu-ULS}. In section 4, we recall
first how the cohomology of the Milnor fiber of a homogeneous
polynomial can be understood in terms of local systems. Then we
apply the result of the previous section to reduce the computation
of the spectrum of a homogeneous polynomial to intersection theory
on a log resolution. In section 5, we prove Theorem
\ref{computationSpectrumFinal} on the combinatorial formula for
the spectrum of hyperplane arrangement by making use of the
explicit intersection theory on the canonical log resolution. In
section 6 we give some examples showing how Theorem
\ref{computationSpectrumFinal} works.

\medskip

We thank M. Saito for sharing with us the preprint \cite{MS} which
was the inspiration for this article. We also thank: A. Dimca, A. Libgober, L.
Maxim, and T. Shibuta for useful discussions.



\section{Notation and Review}

We fix notation and review basic notions, which we need later,
about multiplier ideals, Hodge spectrum, and intersection theory.

\medskip
\noindent {\bf Notation.} By a {\it variety} we will mean a
complex algebraic variety, reduced, and irreducible. For a smooth
variety $X$, the {canonical line bundle} is denoted $\omega_X$ and
we always fix a {canonical divisor}, $K_X$, such that
$\cO_X(K_X)=\omega_X$. Let $\mu:Y\ra X$ be a proper birational
morphism. The exceptional set of $\mu$, denoted by $Ex(\mu)$, is
the set of points $\{y\in Y\}$ where $\mu$ is  not biregular. For
a divisor $D$ on $X$ with support $\text{Supp}(D)$, we say that
$\mu$ is a {\it log resolution} of $(X,D)$ if $Y$ is smooth and
$\mu^{-1}(\text{Supp}(D))\cup Ex(\mu)$ is a divisor with simple
normal crossings. Such a resolution always exists, by Hironaka.
The {\it relative canonical divisor} of $\mu$ is
$K_{Y/X}=K_Y-\mu^*(K_X)$. If $D=\sum_{i\in S} \alpha_iD_i$ is a
divisor on $X$ with real coefficients, where $D_i$ are the
irreducible components of $D$ for $i\in S$, and $\alpha_i\in\bR$,
the {round down} of $D$ is the integral divisor
$\rndown{D}=\sum\rndown{\alpha_i}D_i$. Here, $\rndown{.}$ is the
round-down of a real number. We also use $\{.\}$ to mean the
fractional part of a real number. For $\alpha\in \bR^S$ and a
reduced effective divisor $D=\cup_{i\in S}D_i$, we frequently use
the notation $\alpha\cdot D$ to mean the $\bR$-divisor $\sum_{i\in
S}\alpha_iD_i$.

\medskip
\noindent {\bf Multiplier ideals.} See \cite{La}- Chapter 9 for
more on multiplier ideals. Let $X$ be a smooth   variety. Let $D$
be an effective $\bQ$-divisor on $X$. Let $\mu: Y \ra X$ be a log
resolution of $(X,D)$. The {\it multiplier ideal} of $D$ is the
ideal sheaf
$$\cJ(D) := \mu_* \cO_{Y}(K_{Y/X}-\rndown{\mu^* D})\ \ \ \subset
\cO_X.$$ The choice of a log resolution does not matter in the
definition of $\cJ(D)$. Equivalently, $\cJ(D)$ can be defined
analytically to consist, locally, of all holomorphic functions $g$
such that $|g|^2/\prod_{i\in S}|f_i|^{2\alpha_i}$ is locally
integrable, where $f_i$ are local equations of the irreducible
components of $D$ and $\alpha_i$ their multiplicities. For the
following see \cite{La}- Theorems 9.4.1, 9.4.9.

\begin{thm}\label{vanishing}  With the notation as above, $R^j\mu_*\cO_{Y} (
K_{Y/X} -\rndown{\mu^* D}) =0$, for $j>0$. Assume in addition that
$X$ is projective. Let $L$ be any integral divisor such that $L-D$
is nef and big. Then $H^i(X,\cO_X(K_X+L)\otimes_{\cO_X}\cJ(D))=0$
for $i>0$.
\end{thm}

\medskip
\noindent{\bf Hodge spectrum.} See \cite{Ku}-II.8 for more on
Hodge spectrum. Let $f:(\bC^{n},0)\ra(\bC,0)$ be the germ of a
non-zero holomorphic function. Let  $M_f$ be the Milnor fiber of
$f$ defined as
$$M_f=\{z\in\bC^{n}\ |\ |z|<\epsilon\ {\rm and\ }f(z)=t\}$$
for $0<|t|\ll\epsilon\ll 1$. It will not matter which $t$ is
chosen. The cohomology groups $H^*(M_f,\bC)$ carry a canonical
mixed Hodge structure such that the semisimple part $T_s$ of the
monodromy acts as an automorphism of finite order of these mixed
Hodge structures (see \cite{St77} for $f$ with an isolated
singularity, \cite{Na} and \cite{Sa4} for the general case).
Define for $\alpha\in\bQ$, the {\it spectrum multiplicity} of $f$
at $\alpha$ to be
$$n_\alpha(f):=\sum_{j\in\bZ} (-1)^j \dim
\Gr_F^{\rndown{n-\alpha}}\wti{H}^{n-1+j}(M_f,\bC)_{e^{-2\pi
i\alpha}},$$ where $F$ is the Hodge filtration, and
$\wti{H}^*(M_f,\bC)_\lambda$ stands for the $\lambda$-eigenspace
of the reduced cohomology under $T_s$. The {\it Hodge spectrum} of
the germ $f$ is the fractional Laurent polynomial
$$\text{Sp}(f):= \sum_{\alpha\in\bQ} n_\alpha(f) t^\alpha.$$
It was first defined by Steenbrink (\cite{St77},\cite{St89}). We
are using however a slightly different definition, as in
\cite{Bu-HS}, \cite{BS}.

\begin{prop}\label{rangeofspectrum}(\cite{BS} -Proposition
5.2.) For $\alpha \notin (0,n)$, $n_\alpha(f)=0$
\end{prop}

\begin{cor}\label{newdefSp} $$\text{Sp}(f)=\sum_{\alpha\in\; (0,n)\cap\bQ} \left( \sum_{j\in\bZ} (-1)^j \dim
\Gr_F^{\rndown{n-\alpha}}{H}^{n-1+j}(M_f,\bC)_{e^{-2\pi i\alpha}}
\right )\cdot t^\alpha.$$
\end{cor}

\begin{proof} In other words, we can use usual, instead of
reduced, cohomology in the definition of $n_\alpha(f)$, provided
with restrict to the range $\alpha\in (0,n)$. We have
\begin{equation}\label{reducedCohomo}
\wti{H}^j(M_f,\bC)_\lambda =\left\{
\begin{array}{lr}
H^j(M_f,\bC)_\lambda, & \text{ if }j\ne 0\text{ or }\lambda\ne
1,\\
\text{coker}(H^0(\text{point},\bC)\ra H^0(M_f,\bC)_1 )& \text{ if
otherwise,}
\end{array}
\right.
\end{equation}
where the last map is induced by a constant map $X\ra
\text{point}\ \in X$ (see e.g. \cite{Di2}-p.106). Hence if
$\alpha\notin\bZ$, we are in the first case of
(\ref{reducedCohomo}) and we can replace $\wti{H}^*$ by $H^*$ in
the definition of $\text{Sp}(f)$. Assume $\alpha\in\bZ$. We have
$\Gr_F^jH^0(\text{point},\bC)$ is $0$ if $j\ne 0$ and is $\bC$ if
$j=0$. By the second case of (\ref{reducedCohomo}), we only need
to worry about the case when $\rndown{n-\alpha}=n-\alpha$ is
exactly $0$. But this is ruled out by Proposition
\ref{rangeofspectrum}.
\end{proof}

\medskip

\noindent {\bf Intersection theory.} See \cite{Fu} for more on
intersection theory. Let $X$ be a smooth projective   variety of
dimension $n$. For a vector bundle, or locally free
$\cO_X$-module, of finite rank $r$, $\cE$ on $X$, we denote by
$c_i(\cE)$ the image of the $i$-th Chern class of $\cE$ in
$H^{2i}(X,\bZ)$. For $i\ne\{0,\ldots ,r\}$, $c_i(\cE)=0$, and
$c_0(\cE)=1$. We have the following definitions:
\begin{equation}\label{chernDefs}
\begin{array}{lr}
c(\cE)  =\sum_{i} c_i(\cE)\  \ \ &\text{ (total Chern class), }\\
c_t(\cE)  =\sum_{i} c_i(\cE)t^i\  \ \ &\text{ (Chern polynomial), }\\
x_i \text{ formal symbols }\  :\ \prod_{1\le i\le r} (1+x_it)= c_t(\cE)\ \ \ &\text{ (Chern roots), }\\
ch(\cE)=\sum _{1\le i\le r}\exp (x_i) &\text{ (Chern character), }\\
Q(x)=x/(1-\exp(-x)) &\\
td (\cE)=\prod_{1\le i\le r} Q(x_i) &\text{ (Todd class), }\\
c(X)=c(T_X) &\text{ (total Chern class of }X),\\
Td (X)=Td(T_X) &\text{ (Todd class of }X).\\
\end{array}
\end{equation}

The meaning of the Chern roots is the following. The coefficients
of powers of $t$ in $\prod_{1\le i\le r} (1+x_it)$ are elementary
symmetric functions in $x_1,\ldots, x_r$ and they are set to equal
the Chern classes $c_j(\cE)$. Any other symmetric polynomial, such
as the homogeneous terms of fixed degree in the Taylor expansion
of $ch(\cE)$ or $td(\cE)$, can be expressed in terms of elementary
symmetric functions, hence in terms of the $c_j(\cE)$. See
\cite{Fu} -Examples 3.2.3, 3.2.4.

Let $0\ra\cE'\ra\cE\ra\cE''\ra 0$ be an exact sequence of vector
bundles, let $x_1,\ldots, x_r$ be the Chern roots of $\cE$, and
let $\cF$ another vector bundle with Chern roots $y_1,\ldots ,
y_s$. Then (\cite{Fu} -Section 3.2):
\begin{equation}\label{chernFromulas}
\begin{array}{ll}
c_i(\cE^\vee)=(-1)^ic_i(\cE),\ \ \ \  & c_t(\cE) =c_t(\cE')\cdot c_t(\cE''),\\
c_t(\bigwedge^p\cE)=\prod_{i_1<\ldots <i_p}(1+(x_{i_1}+\ldots x_{i_p})t),\ \ \ \ & ch(\cE) =ch(\cE')+ch(\cE''),\\
c_t(\cE\otimes\cF)=\prod_{i,j} (1+(x_i+y_j)t),\ \ \ \ & ch(\cE\otimes\cF) =ch(\cE)\cdot ch(\cF).\\
\end{array}
\end{equation}
For every element of the Grothendieck group of coherent sheaves on
$X$ there are well defined Chern classes. For an algebraic class
$\xi$ in the ring $H^*(X,\bZ)$, let $\xi_j\in H^{2j}(X,\bZ)$
denote the degree $2j$ part of $\xi$, such that $\xi=\sum_j
\xi_j$.

\begin{thm}\label{thm HRR} (Hirzebruch-Riemann-Roch, \cite{Fu}- Corollary
15.2.1) Let $\cE$ be a vector bundle on a smooth projective
  variety $X$ of dimension $n$. Then $\chi (X,\cE)$ is the
intersection number $(ch(\cE)\cdot Td(X))_n$.
\end{thm}

\section{Hodge filtration for local systems}

Recall (e.g. from \cite{Di}) that a {\it complex local system}
$\cV$ on a complex manifold $X$ is a locally constant sheaf of
finite dimensional complex vector spaces. The rank of $\cV$ is the
dimension of a fiber of $\cV$. Local systems of rank one on $X$
are equivalent with representations $H_1(X,\bZ)\ra \bC^*$. Unitary
local systems of rank one on $X$ correspond to representations
$H_1(X,\bZ)\ra S^1$, where $S^1$ is the unit circle in $\bC$. For
a smooth variety $X$, local systems are defined on the
corresponding complex manifold.

\medskip

Let $X$ be a smooth   projective variety of dimension $n$. Let $D$
be a reduced effective divisor on $X$ with irreducible
decomposition $D=\cup_{i\in S} D_i$, for a finite set of indices
$S$. Let $U=X-D$ be the complement of $D$ in $X$. Rank one unitary
local systems on $U$ have the following geometric interpretation.
Define first the group of {\it realizations of boundaries} of $X$
on $D$
$$\Pic^\tau(X,D):=\left\{ (L,\alpha) \in\Pic(X)\times [0,1)^S :
c_1(L)=\sum_{i\in S}\alpha_i\cdot [D_i]\ \in H^2(X,\bR)\right\},$$
where the group operation is
$$(L_,\alpha)\cdot
(L',\alpha')=\left (L\otimes L'\otimes\cO_X(-
\rndown{\alpha+\alpha'}\cdot D)), \{\alpha+\alpha'\}\right ).$$
Here $\alpha\cdot D$ means the divisor $\sum_{i\in S}\alpha _i
D_i$, and $\rndown{.}$ (resp. $\{.\}$) is taking the round-down
(resp. fractional part) componentwise. Note that the inverse of
$(L,\alpha)$ is $(M,\beta)$ where
$M=L^\vee\otimes\cO_X(\sum_{\alpha_i\ne 0}D_i)$, and $\beta_i$ is
$0$ if $\alpha_i=0$ and is $1-\alpha_i$ otherwise.

\begin{thm}\label{ls}(\cite{Bu-ULS} - Theorem 1.2.) Let $X$ be a
smooth   projective variety, $D$ a divisor on $X$, and let
$U=X-D$. There is a natural canonical group isomorphism
$\Pic^\tau(X,D)\xrightarrow{\sim}\Hom (H_1(U,\bZ),S^1)$ between
realizations of boundaries of $X$ on $D$ and unitary local systems
of rank one on $U$.
\end{thm}

Fix a log resolution $\mu:Y\ra X$ of $(X,D)$ which is an
isomorphism above $U$. Let $E=Y-U$ with irreducible decomposition
$E=\cup_{j\in S'}E_j$.

\begin{prop}\label{comparison}(\cite{Bu-ULS} - Proposition 3.3.)
The map $\Pic^\tau(X,D) \ra\Pic^\tau(Y,E)$ given by $(L,\alpha)
\mapsto (\mu^*L-\rndown{e}\cdot E, \{e\})$ is an isomorphism,
where $e\in \bR^{S'}$ is given by $\mu^*(\alpha\cdot D)=e\cdot E$.
\end{prop}

For a unitary local system $\cV$ on $U$, denote by
$\overline{\cV}$ the vector bundle on $Y$ given by the canonical
Deligne extension of $\cV$ to $Y$ (see \cite{De}). The relation
between canonical Deligne extensions and realizations of
boundaries, and the explicit isomorphism of Theorem \ref{ls} in
the case of the complement of a simple normal crossings divisor is
the following:

\begin{lem}\label{relationboundaries}(\cite{Bu-ULS}-Proof of Theorem 1.2 and Remark 8.2
(a).) With the notation as in Proposition \ref{comparison}, let
$\cV\in\Hom (H_1(U,\bZ),S^1)$ be a rank one unitary local system
on $U$. Then $\cV$ corresponds to $(M,\beta)\in\Pic^\tau(Y,E)$
where $M=\overline{\cV}\otimes\cO_Y(\sum_{\beta_j\ne 0}E_j)$ and
$\beta_j\in [0,1)$ is such that the monodromy of $\cV$ around a
general point of $E_j$ is multiplication by $\exp(2\pi i\beta_j)$.
\end{lem}

Unitary local systems admit a canonical Hodge filtration $F$ on
cohomology such that:

\begin{thm}\label{Hodgeforlocsys}(\cite{Ti}-2nd Theorem, part
(a).) With notation as in Proposition \ref{comparison}, let $\cV$
be a unitary local system on $U$.  Then $$\dim
\Gr_F^pH^{p+q}(U,\cV) = h^q(Y,\Omega_Y^p(\log E)\otimes
\overline{\cV} ).$$
\end{thm}

The main result of this section describes the pieces of the Hodge
filtration on the cohomology of unitary rank one local systems on
complements to arbitrary divisors. It generalizes \cite{Bu-ULS} -
Proposition 6.4 and part of \cite{DS}- Theorem 2. It is also related to \cite{Bu-ULS}- Theorems 1.3, 1.4, and \cite{L07}- Theorem 2.1.

\begin{thm}\label{mainThmsectiononLocSys}  Let $X$ be a smooth
projective variety of dimension $n$, $D$ a divisor on $X$, and
$U=X-D$. Let $\cV\in \Hom (H_1(U,\bZ),S^1)$ be a rank one unitary
local system on $U$ corresponding to $(L,\alpha)\in \Pic^\tau
(X,D)$. Let $\mu :(Y,E)\ra (X,D)$ be a log resolution which is an
isomorphism above $U$. Then:

(a) $$\dim \Gr_F^pH^{p+q}(U,\cV^\vee) =$$
$$= h^{n-q}(Y,\Omega_Y^p(\log E)^\vee\otimes\omega_Y\otimes\mu^*
L\otimes\cO_Y(-\rndown{\mu^*(\alpha\cdot D)}));$$

(b) if $\alpha _i\ne 0$ for all $i\in S$,
$$\dim \Gr_F^pH^{p+q}(U,\cV) = $$
$$= h^q(Y, \Omega_Y^{n-p}(\log E)^\vee\otimes \omega_Y\otimes\mu^*
 L\otimes\cO_Y(-\rndown{\mu^*((\alpha-\epsilon)\cdot D)} )),$$
for all $0<\epsilon\ll 1$.

In particular, in terms of multiplier ideals:

(c) $\dim \Gr_F^0H^q(U,\cV^\vee) = h^{n-q}(X,\omega_X\otimes
L\otimes\cJ(\alpha\cdot D));$

(d) if $\alpha _i\ne 0$ for all $i\in S$, for all $0<\epsilon\ll
1$,
$$\dim \Gr_F^nH^{n+q}(U,\cV) = h^q(X,\omega_X\otimes L\otimes\cJ((\alpha-\epsilon)\cdot
D)),$$ which is $0$ if $q\ne 0$.

\end{thm}

\medskip

\begin{pf} (a) By Theorem \ref{Hodgeforlocsys},
$$\dim \Gr_F^pH^{p+q}(U,\cV^\vee) = h^q(Y,\Omega_Y^p(\log
E)\otimes \overline{\cV^\vee} ).$$ By Proposition
\ref{comparison}, $\cV$ corresponds to
$(\mu^*L\otimes\cO_Y(-\rndown{\mu^*(\alpha\cdot D)}),\beta)$ in
$\Pic^\tau(Y,E)$, where $\beta_j$ is the fractional part of the
coefficient of $E_j$ in $\mu^*(\alpha\cdot D)$. By calculating the
inverse in the group $\Pic^\tau(Y,E)$, the dual local system
$\cV^\vee$ corresponds to
$(\mu^*L^\vee\otimes\cO_Y(\rndown{\mu^*(\alpha\cdot
D)}+\sum_{\beta_j\ne 0}E_j),\gamma)$ in $\Pic^\tau(Y,E)$, where
$\gamma_j$ is $0$ if $\beta_j=0$ and is $1-\beta_j$ if $\beta_j\ne
0$. Hence, by Lemma \ref{relationboundaries},
\begin{align*}
\overline{\cV^\vee}
&=\mu^*L^\vee\otimes\cO_Y(\rndown{\mu^*(\alpha\cdot
D)}+\sum_{\beta_j\ne 0}E_j -\sum_{\gamma_j\ne 0}E_j)\\
&=\mu^*L^\vee\otimes\cO_Y(\rndown{\mu^*(\alpha\cdot D)}).
\end{align*}
The conclusion follows by Serre duality.

(b) We have
$$\dim \Gr_F^pH^{p+q}(U,\cV) =
h^q(Y,\Omega_Y^p(\log E)\otimes \overline{\cV} ).$$ By Lemma
\ref{relationboundaries},
$\overline{\cV}=\mu^*L\otimes\cO_Y(-\rndown{\mu^*(\alpha\cdot
D)}-\sum_{\beta_j\ne 0}E_j )$. Also, we use the isomorphism
$\Omega_Y^p(\log E)\cong \Omega_Y^{n-p}(\log
E)^\vee\otimes\omega_Y\otimes\cO_Y(\sum_jE_j)$ (see \cite{EV}-6.8
(b)). If $\alpha_i\ne 0$ for all $i$, that is if $\cV$ is not the
restriction to $U$ of a local system over a larger open subset of
$X$, then the coefficients of $E_j$ in $\mu^*(\alpha\cdot D)$ are
nonzero. Hence
$$-\rndown{\mu^*(\alpha\cdot D)}+\sum_{\beta_j=0}E_j = -\rndown{\mu^*((\alpha-\epsilon)\cdot D))}, $$
for all $0<\epsilon\ll 1$. The conclusion follows.

(c) We let $p=0$ in (a) and use the definition of multiplier
ideals. The identification of $H^{n-q}$ of the $\cO_Y$-module from
(a) with $H^{n-q}$ of the $\cO_X$-module $\omega_X\otimes
L\otimes\cJ(\alpha\cdot D)$ is due to the triviality of the Leray
spectral sequence that follows from the projection formula
(\cite{Ha} -III.8 Ex. 8.3) and the first part of Theorem
\ref{vanishing}.

(d) We let $p=n$ in (b), then proceed as in (c). The vanishing for
$q\ne 0$ follows from the second part of Theorem \ref{vanishing}.
$\ \ \ \ \ \square$

\end{pf}

\begin{rem}\label{genofDimcaSaito} (i) Part (a) of Theorem
\ref{mainThmsectiononLocSys} generalizes part of \cite{DS}-
Theorem 2. They proved that when $X=\bP^n$ and $D$ is a
hypersurface in $X$,
$$F^nH^n(U,\bC)=H^0(X,V^0(\omega_X(D))).$$ The sheaf
$V^0(\omega_X(D))$ has a description in terms of multiplier ideals
via a result of \cite{BS}. More precisely,
$$V^0(\omega_X(D)) = \omega_X\otimes\cO_X(D)\otimes \cJ((1-\epsilon)D),$$
for all $0<\epsilon\ll 1$. To see that this indeed follows from
part (a) of Theorem \ref{mainThmsectiononLocSys}, we have
isomorphisms:
$$ F^nH^n(U,\bC) = \Gr_F^nH^n(U,\bC) =$$
$$ = H^n(Y,\Omega_Y^n(\log E)^\vee\otimes\omega_Y) =
H^0(Y,\omega_Y\otimes\cO_Y(E)) =$$
$$= H^0(Y,\omega_Y\otimes\mu^*\cO_X(D)\otimes\cO_Y(\rndown{\mu^*((1-\epsilon)D)})) =$$
$$= H^0(X,V^0(\omega_X(D))).$$

(ii) Examples of Hodge numbers  of local systems can be found in \cite{Bu-ULS}- Example 6.6, \cite{L01}, \cite{L07}- Section 6.

\end{rem}

\section{Milnor fibers and local systems}

We recall first how the cohomology of the Milnor fiber of a
homogeneous polynomial can be understood in terms of local
systems. Then we apply the result of the previous section to
reduce the computation of the spectrum of a homogeneous polynomial
to intersection theory on a log resolution.

\medskip

Let $f\in\bC[x_1,\ldots ,x_n]$ be a homogeneous polynomial of
degree $d$. Let $f=\prod_{i\in S}f_i^{m_i}$ be the irreducible
decomposition of $f$, and $d_i$ be the degree of $f_i$. Denote by
$D$ (resp. $D_i$) the hypersurface defined by $f$ (resp. $f_i$) in
$X:=\bP^{n-1}$. Let $U=X-D$.

\medskip

The {\it global Milnor fiber} of $f$ is
$M:=f^{-1}(1)\subset\bC^n$. The {\it geometric monodromy} is the
map $h:M\ra M$ given by $a\mapsto e^{2\pi i/d}\cdot a$. It is
known that $H^i(M,\bC)=H^i(M_f,\bC)$, where $M_f$ is the Milnor
fiber of the germ of $f$ at $0\in\bC^n$, such that $h^*$ on
$H^i(M,\bC)$ corresponds to the monodromy action $T$ on
$H^i(M_f,\bC)$. Hence the monodromy $T$ is diagonalizable and the
eigenvalues are $d$-th roots of unity. See e.g. \cite{Di2}-p.72.
Also, the Hodge filtration $F$ on $H^i(M_f,\bC)$ is induced by the
one on $H^i(M,\bC)$. This fact seems well known to experts, and it
has been used for example in \cite{St89}- Theorem 6.1. However we
could not find a reference. Since a proof of this fact would not
be elementary and take us too far, we take as the definition of
$F$ the canonical Hodge filtration of $H^i(M,\bC)$ .

\medskip

The group $G:=<h>=\bZ/d\bZ$ acts on $M$ freely and the quotient
$M/G$ can be identified with $U$. Let $p:M\ra U$ be the covering
map. Write
$$p_*\bC_M =\oplus _{k=1}^{d} \cV_k,$$
where $\cV_k$ is the rank one unitary local system on $U$ given by
the $e^{-2\pi i k/d}$-eigenspaces of fibers of the local system
$p_*\bC_M$. Then, since $p$ is finite, by Leray spectral sequence
one has for $1\le k\le d$ (see also \cite{CS} -Theorem 1.6):
\begin{equation}\label{MilnFibLocSYs}
H^i(M,\bC)_{e^{-2\pi i k/d}}=H^i(U,\cV_k).
\end{equation}
This isomorphism preserves the Hodge filtration by the
functoriality of the Hodge filtration for unitary local systems
(see \cite{Ti}-\S 6). Thus, in the case of homogeneous
polynomials, the computation of the Hodge filtration on the
cohomology of the Milnor fiber is reduced to the computation of
the Hodge filtration on the cohomology of unitary rank one local
systems on the complement of the projective hypersurface.

\medskip

Next result is well known to experts. We give a proof since we
could not find a reference.

\begin{lem}\label{monodromyGoesAround} With notation as above, the monodromy of $\cV_k$ around a
general point of  $D_i$ is given by multiplication by $e^{2\pi
ikm_i/d}$.
\end{lem}
\begin{proof}
The $\cV_k$, with tensor product, form a group isomorphic to  $G$
(e.g. \cite{Bu-ULS}-\S 5). So it is enough to prove the lemma for
$k=1$.

Fix $i$ and denote $m_i$ by $m$. Let $P$ be a general point of
$D_i$. We consider a small loop $\tau=\{Q_\theta\ \in U|\ \theta
\in [0,1]\}$ around $P$, with $Q_0=Q_1=:Q$. We need to look at the
action $T_i$ of going along $\tau$ counterclockwise on the fiber
$(\cV_1)_Q$.

\medskip

First, $(p_*\bC _M)_Q=\oplus_{1\le j\le d} \bC v_{\xi ^j \bx}$,
where $\bx\in M\subset\bC^n$ is fixed and such that $p(\bx)=Q$,
$\xi=e^{-2\pi i/d}$, and $v_{\xi ^j \bx}$ are linearly
independent. Here $\{ \xi ^1 \bx ,\ldots ,\xi ^d \bx  \}$ is
$p^{-1}(Q)$. The action induced by $h$ on $(p_*\bC _M)_Q$ is given
by $v_{\xi ^j \bx}\mapsto v_{\xi ^{j-1} \bx}$. Hence, its
$\xi$-eigenspace is
$$
(\cV_1)_Q=\{ a\cdot\sum_{1\le j\le d}\xi^j v_{\xi ^j \bx}\ |\
a\in\bC\ \}.
$$

We will show that $T_iv_{\xi ^j \bx}=v_{\xi ^{j+m} \bx}$. This
implies that $T_i$ acts on $(\cV_1)_Q$ via multiplication by
$\xi^{-m}$, which is what we wanted to show.

\medskip

By considering the loop $\tau$ lying in a (real) plane, we can
simplify the computation. To this end, after linear change of
coordinates, we can assume the following. First, we can assume
$i=1$. Let $V=\{x_3=\ldots =x_n=0\}\subset \bC^n$ be transversal
to all the hyperplanes $\{f_i=0\}$. Define $M':=M\cap V$,
$U':=U\cap \bP V\subset \bP V=\bP^1$,
$f'(x_1,x_2):=f(x_1,x_2,0,\ldots ,0)$, and $p'=p_{|M'}$. Then
$M'=f'^{-1}(1)$. We can assume $P=[0:1]$ and $Q_\theta=[e^{2\pi i
\theta}:1]$ in $\bP V=\bP^1$, and that $f'=x_1^m g$ with $g$
having no zeros inside the disc centered at $P$ with boundary
$\tau=\{Q_\theta\}$. Then
$$
p'^{-1}(Q_\theta)=\{ \lambda\cdot(e^{2\pi i\theta},1)\ \in\ \bC^2
\ |\ \lambda ^de^{2\pi i m\theta}g(e^{2\pi i\theta},1)=1\  \}.
$$
For each $\theta$ fix $a_\theta$ such that $a_\theta^d=g(e^{2\pi
i\theta},1)^{-1}$. We can assume $a_0=a_1$. Then
$$
p'^{-1}(Q_\theta)=\{  \bx_{j,\theta}:= e^{2\pi
i(j-m\theta)/d}a_\theta \cdot(e^{2\pi i\theta},1)\ |\ 1\le j\le d\
\}.
$$
Fix $j$ and let $\bx =\bx_{j,0}$. Starting at $\bx$, going
counterclockwise along the inverse image by $p'$ of $\tau$, we end
up at $\bx_{j,1}=e^{-2\pi i m/d}\bx$. This shows that $T_iv_{\xi
^j \bx}=v_{\xi ^{j+m} \bx}$.
\end{proof}

\begin{lem}\label{explicitV_k} With notation as in Lemma \ref{monodromyGoesAround},
let $(L^{(k)},\alpha^{(k)})\in \Pic^\tau(X,D)$ correspond to
$\cV_k$ under the isomorphism of Theorem \ref{ls}. Then:
$$
\alpha^{(k)}_i  = \left\{\frac{km_i}{d} \right\},\ \ \ \ \ L^{(k)}
=\cO_X\left ( \mathop{\sum_{i\in S}}_{} \alpha_i^{(k)} d_i \right
)
 .
$$
\end{lem}

\begin{proof} By Proposition \ref{comparison} and Lemma
\ref{relationboundaries}, $\alpha^{(k)}_i\in [0,1)$ is given by
the monodromy of $\cV_k$ around a general point of (the proper
transform of) $D_i$. The conclusion for $\alpha^{(k)}_i$ then
follows from Lemma \ref{monodromyGoesAround}. The condition that
$(L^{(k)},\alpha^{(k)})\in \Pic^\tau(X,D)$ is that the degree of
$L^{(k)}$ equals $\sum_{i\in S}\alpha^{(k)}_i d_i$.
\end{proof}
Alternatively, one can prove Lemma \ref{explicitV_k} using
\cite{Bu-ULS} -Corollary 1.10 .

\medskip

Now we draw some conclusions about the spectrum $\text{Sp}(f)$ of
a homogeneous polynomial at the origin. By above discussion and
Corollary \ref{newdefSp}, the only rational numbers which can have
nonzero multiplicity in $\text{Sp}(f)$ are of the type
\begin{equation}\label{alpha}
\alpha =\frac{k}{d}+p\ \in\ (0,n),\ \ \ \ \text{ with } k,p
\in\bZ,\ 1\le k\le d,\ 0\le p <n.
\end{equation}

Let $\mu:(Y,E)\ra (X,D)$ be a log resolution which is an
isomorphism above $U$. With $\alpha, k,p,$ as in (\ref{alpha}),
define
\begin{align*}
& \beta_i^{(k)} :=\left\{ -\frac{km_i}{d} \right\}, \ \ \ \ \ \ \
\ \ M^{(k)} :=\cO_X\left ( \mathop{\sum_{i\in S}}_{}
\beta_i^{(k)} d_i  \right ),\\
&\cE_\alpha  := \Omega^{n-p-1}_{Y}(\log E)^\vee \otimes
\omega_{Y}\otimes \mu^*M^{(k)} \otimes \cO_{Y} \left(
-\rndown{\mu^*\left(  \beta^{(k)}\cdot D_{red} \right )} \right ),
\end{align*}
where in the last sheaf the tensor products are over $\cO_{Y}$.

\begin{prop}\label{spectrumHomogenous}
Let $\alpha$ be as in (\ref{alpha}). The multiplicity of $\alpha$
in $\text{Sp}(f)$ is
$$n_\alpha (f)=(-1)^{n-p-1}\chi(Y,\cE_\alpha).$$
\end{prop}

\begin{proof} By Corollary \ref{newdefSp} and
(\ref{MilnFibLocSYs}),
$$n_\alpha (f)= \sum_{j\in\bZ} (-1)^j \dim
\Gr_F^{n-p-1}{H}^{n-1+j}(U,\cV_k)=\star.$$ Now,
$(M^{(k)},\beta^{(k)})$ is an element of $\Pic^\tau (X,D)$ and it
can be checked by using Lemma \ref{explicitV_k} that it
corresponds to the dual $\cV_k^\vee$. In fact, $\cV_k^\vee$ equals
$\cV_{d-k}$ if $k\ne d$, and equals $\cV_d$ if $k=d$. Since
$(\cV_k^\vee)^\vee=\cV_k$, by applying Theorem
\ref{mainThmsectiononLocSys}-(a), we have
$$\star =\sum_{j\in\bZ} (-1)^{j} h^{n-j-p-1}(Y,\cE_\alpha),
$$
which is equivalent to what we claimed.
\end{proof}

By Hirzebruch-Riemann-Roch (Theorem \ref{thm HRR}), Proposition
\ref{spectrumHomogenous} is useful when the topology of a log
resolution is known:

\begin{cor}\label{HomogRR} Let $\alpha$ be as in (\ref{alpha}). The multiplicity of $\alpha$
in $\text{Sp}(f)$ is the intersection number
$$n_\alpha (f)=(-1)^{n-p-1}\left ( ch(\cE_\alpha)\cdot Td(Y)\right )_{n-1}.$$
\end{cor}

\section{Spectrum of hyperplane arrangements}

 A {\it central
hyperplane arrangement} in $\bC^n$ is a finite set $\cA$  of
vector subspaces of dimension $n-1$. The {\it intersection
lattice} of $\cA$, denoted $L(\cA)$, is the set of subspaces of
$\bC^n$ which are intersections of subspaces $V\in\cA$ (see
\cite{OT}). For $V\in\cA$, let $f_V$ be the linear homogeneous
equation defining $V$, and let $m_V\in \bN-\{0\}$. Let
$f=\prod_{V\in\cA}f_V^{m_V}\in\bC[x_1,\ldots ,x_n]$ be a
homogeneous polynomial of degree $d=\sum_{V\in\cA}m_V$. Denote by
$D$  the hypersurface defined by $f$  in $X:=\bP^{n-1}$. Let
$U=X-D$. Let $\cG'\subset L(\cA)-\{\bC^n\}$ be a building set (see
\cite{DP}-2.4 or \cite{Te}-Definition 1.2). Let
$\cG=\cG'\cup\{0\}$. For simplicity, one can stick with the
following example for the rest of the article:
$\cG=L(\cA)\cup\{0\}-\{\bC^n\}$, when $\cG'$ is chosen to be
$L(\cA)-\{\bC^n\}$. The advantage of considering smaller building
sets is that computations might be faster (see \cite{Te}-Example
1.3-(c)). For any vector space $V$ of $\bC^n$, we denote by
$\delta(V)$ (resp. $r(V)$) the dimension (resp. codimension) of
$V$.

\medskip
\noindent {\bf The canonical log resolution.} We consider the
canonical log resolution $\mu:(Y,E)\ra (X,D)$ of $(X,D)$ obtained
from successive blowing-ups of the (disjoint) unions of (the
proper transforms) of $\bP(V)$ for $V\in \cG-\{0\}$ of same
dimension. This is the so-called wonderful model of \cite{DP}-
section 4. More precisely, $\mu$ and $Y$ are constructed as
follows (see also \cite{Bu-JNHA} -Section 4, \cite{MS}). Let
$X_0=X$. Let $C_0$ be the disjoint union of $\bP(V)$ for
$V\in\cG-\{0\}$ with $\delta (V)=1$. Let $\mu_0:X_1\ra X_0$ be the
blow up of $C_0$. Then $\mu_i$ and $X_{i+1}$ are constructed
inductively as follows. Let $C_i\subset X_{i}$ be the disjoint
union of the proper transforms, under the map $\mu_{i-1}$, of
$\bP(V)$ for $V\in \cG-\{0\}$ with $\delta(V)=i+1$.  Let
$\mu_i:X_{i+1}\ra X_i$ for $0\le i <n-2$ be the blow up of $C_i$.
Define $Y=X_{n-2}$ and $\mu$ as the composition of the $\mu_i$.

\medskip

For $V\in \cG-\{0\}$ with $\delta(V)=i+1$, let $E_V$ be the proper
transform of the exceptional divisor in $X_{i+1}$ corresponding to
(the proper transform of) $\bP(V)$ (in $X_i$). Also  let $E_0$
denote the proper transform in $Y$ of a general hyperplane of $X$.
Denote by $[E_V]$ the cohomology class of $E_V$ on $Y$
($V\in\cG$), where it will be clear from context what coefficients
(integral, rational) we are considering. $E$ denotes the union of
$E_V$ for $V\in\cG-\{0\}$.

\medskip
\noindent{\bf Intersection theory on the canonical log
resolution.} Let $I\subset\bZ[c_V]_{V\in\cG}$ be the ideal
generated by two types of polynomials:
\begin{equation}\label{eq, type 1}
\prod_{V\in \cH}c_V
\end{equation}
if $\cH\subset\cG$ is not a nested subset, and by
\begin{equation}\label{eq, type 2}
\prod_{V\in \cH}c_V\left (  \sum_{W'\subset W} c_{W'} \right
)^{d_{\cH,W}},
\end{equation}
where $\cH\subset \cG$ is a nested subset, $W\in\cG$ is such that
$W\varsubsetneq V$ for all $V\in\cH$, and
$d_{\cH,W}=\delta(\cap_{V\in\cH} V) -\delta (W)$. In (\ref{eq,
type 2}), one considers $\cH=\emptyset$ to be nested, in which
case (\ref{eq, type 2}) is defined for every $W\in\cG$ by setting
$\delta (\emptyset)=n$. Here $I$ is the ideal of \cite{DP}-5.2, for the projective case.
$I$ depends only on $\cG$ and  $\bZ[c_V]_{V\in\cG}/ I$ is
isomorphic to the cohomology ring of the canonical log resolution:

\begin{thm}\label{rem. isom} (\cite{DP}) With notation as above, there is an isomorphism
\begin{align}\label{eq. isom cohomology}
\bZ[c_V]_{V\in\cG} / I\   & \mathop{\longrightarrow}^{\sim}\
H^*(Y,\bZ)\  \mathop{\longleftarrow}^{\sim}\
\bZ[[c_V]]_{V\in\cG} / I \\
\notag & 1\mapsto [Y],\\
\notag & c_V  \mapsto [E_V]\ \ \ \ \text{ if }V\ne 0,\\
\notag & c_0 \mapsto -[E_0].
\end{align}
\end{thm}
Theorem \ref{rem. isom} follows from \cite{DP}-5.2 Theorem,
\cite{DP}-4.1 Theorem, part (2), and \cite{DP}-4.2 Theorem, part
(4); see also \cite{Bu-JNHA} -Remark 4.3. Remark that the degree
$n-1$ homogeneous part of $\bZ[[c_V]]_{V\in\cG}/I$ can be
identified with $\bZ\cdot (-c_0)^{n-1}$.

\medskip
For every $V\in\cG-\{0\}$ define a formal power series $F_V\in\bZ
[[c_V]]_{V\in\cG}$ by
$$
F_V:=(1-\mathop{\sum_{ W\varsubsetneq
V}}_{W\in\cG}c_W)^{-r(V)}(1+c_V)(1-\mathop{\sum_{ W\subset V
}}_{W\in\cG} c_W)^{r(V)}.$$  Also, set $F_0=(1-c_0)^n$ and define
$F:=\prod_{V\in\cG}F_V$.

\begin{prop}\label{propchernclasscanonicalresolution}(\cite{Bu-JNHA} -Proposition 4.7.)
 The total Chern class $c(Y)$ is the image
in $H^*(Y,\bZ)$ of $F$ under the map (\ref{eq. isom cohomology}).
\end{prop}

Let $Q(x)$ be as in (\ref{chernDefs}). For every $V\in\cG-\{0\}$
define a formal power series $G_V\in\bQ [[c_V]]_{V\in\cG}$ by
$$G_V:=Q(-\mathop{\sum_{ W\varsubsetneq
V}}_{W\in\cG}c_W)^{-r(V)}Q(c_V)Q(-\mathop{\sum_{ W\subset
V}}_{W\in\cG} c_W)^{r(V)}.$$ Also, set $G_0=Q(-c_0)^n$ and define
$G:=\prod_{V\in\cG}G_V$.

\begin{cor}\label{cor. todd class canonical resolutions} (\cite{Bu-JNHA} -Corollary 4.8.)
The Todd class $Td(Y)$ is the image in $H^*(Y,\bQ)$ of $G$ under
the map induced by (\ref{eq. isom cohomology}) after
$\otimes_\bZ\bQ$.
\end{cor}

For a power series $\xi \in\bZ[[c_V]]_{V\in\cG}$, let $\xi_i$
denote the degree $i$ part, such that $\xi=\sum_i\xi_i$. Define a
formal power series $H\in\bZ[[c_V]]_{V\in\cG}$ by
$$H:=\left ( \sum_i (-1)^iF_i\right
) \cdot \prod_{V\in\cG-\{0\}}\frac{1}{1-c_V}.
$$

\begin{lem}\label{computationOmegaLog}
The total Chern class $c(\Omega^1_{Y}(\log E))$ is the image in
$H^*(Y,\bZ)$ of $H$ under the map (\ref{eq. isom cohomology}).
\end{lem}
\begin{proof} $\Omega^1_{Y}(\log E)$ fits into a
short exact sequence (see \cite{EV} -2.3 Properties (a)):
$$0\ra \Omega^1_{Y}\ra \Omega^1_{Y}(\log E)\ra
\bigoplus _{V\in\cG-\{0\}} \cO_{E_V}\ra 0.
$$
By (\ref{chernFromulas}), $c(\Omega^1_{Y}(\log
E))=c(\Omega^1_{Y})\cdot\prod_{V\in\cG-\{0\}}c(\cO_{E_V})$. Now,
by Proposition \ref{propchernclasscanonicalresolution},
$c_i(\Omega^1_{Y})=(-1)^ic_i(T_X)=(-1)^i(F)_i$. Also,
$c(\cO_{E_V})=1/(1-[E_V])$ since $E_V$ is a hypersurface in $Y$.
\end{proof}

Fix $p\in\{0,\ldots ,n-1\}$. Denote by $e_i(x_1,\ldots ,x_{n-1})$
the coefficient of $t^i$ in $\prod_{1\le i\le n-1}(1+x_it)$. The
coefficient of $t^i$ in $\prod_{1\le i_1<\ldots <i_p\le
n-1}(1+(x_{i_1}+\ldots +x_{i_p})t)$ is $K_{p,i}'(e_1,\ldots
,e_{n-1})$ for some polynomial $K_{p,i}'$ in $n-1$ variables over
$\bZ$. Here $K_{0,i}'=1$ if $i=0$ and equals $0$ if $i\ne 0$.
Define
$$K_{p,i}:=K_{p,i}'(H_1,\ldots
,H_{n-1})\in \bZ[c_V]_{V\in\cG},
$$
where $H_j$ is the degree $j$ part of $H$.

\begin{lem}\label{computationWedgeOmegaLog} The Chern class $c_i(\Omega^p_{Y}(\log E))$ is the image
in $H^*(Y,\bZ)$ of the polynomial $K_{p,i}$ under the map
(\ref{eq. isom cohomology}).
\end{lem}
\begin{proof} Since $\Omega^p_{Y}(\log E)=\bigwedge^p\Omega^1_{Y}(\log
E)$, the claim follows from (\ref{chernFromulas}) and Lemma
\ref{computationOmegaLog}.
\end{proof}

For $1\le p\le n-1$, the degree $j$ term  in the Taylor expansion
of $\sum_{1\le i\le p}e^{x_i}$ is $P_{p,j'}(e_1,\ldots, e_{p})$
for some polynomial $P_j'$ in $p$ variables over $\bQ$. Let
$P_{p,j}\in\bQ[c_V]_{V\in\cG}$ be
$$P_{p,j}:=P_{p,j'}(-K_{p,1},\ldots, (-1)^iK_{p,i},\ldots (-1)^{p}K_{p,p}).
$$
Define $P_p:=\sum_i P_{p,i}$. For $p=0$ set $P_0=1$. Then by
(\ref{chernFromulas}) and Lemma \ref{computationWedgeOmegaLog} we
have:
\begin{lem}\label{computationChernWedgeOmegaLogDual} The Chern
character $ch(\Omega^p_{Y}(\log E)^\vee)$ is the image in
$H^*(Y,\bQ)$ of  $P_p$ under the map induced by (\ref{eq. isom
cohomology}) after $\otimes_\bZ\bQ$.
\end{lem}

\medskip
\noindent {\bf Computation of spectrum.} Now we complete the
computation of the Hodge spectrum $\text{Sp}(f)$ of $f$ at the
origin. The only rational numbers $\alpha$ which can appear in
$\text{Sp}(f)$ are of the type (\ref{alpha}), i.e.
\begin{equation*}
\alpha =\frac{k}{d}+p\ \in\ (0,n),\ \ \ \ \text{ with } k,p
\in\bZ,\ 1\le k\le d,\ 0\le p <n.
\end{equation*}
where $d$ is the degree of $f$. By Corollary \ref{HomogRR}, the
multiplicity of $\alpha$ in $\text{Sp}(f)$ is the intersection
number
\begin{equation}\label{lastRR}
n_\alpha (f)=(-1)^{n-p-1}\left ( ch(\cE_\alpha)\cdot Td(Y)\right
)_{n-1},
\end{equation}
where $\cE_\alpha$ is defined as follows. For $V\in\cA$, let $m_V$
be the multiplicity of the irreducible component $V$ of
$f^{-1}(0)$. Let
$$
\beta_V^{(k)} :=\left\{ -\frac{km_V}{d} \right\},\ \ \ \ \ \ \ \ \
M^{(k)} :=\cO_X\left( \sum_{V\in\cA} \beta_V^{(k)} \right) ,$$
$$
\cL^{(k)}:=\omega_{Y}\otimes \mu^*M^{(k)} \otimes \cO_{Y} \left(
-\rndown{\mu^*\left( \beta^{(k)}\cdot D_{red}\right )} \right ).
$$
Now define
$$
\cE_\alpha  := \Omega^{n-p-1}_{Y}(\log E)^\vee \otimes \cL^{(k)}.
$$
We also need to fix some more notation. For $\beta\in \bQ^\cA$ and
$V\in \cG$, let
$$
s_V(\beta):= \sum_{V\subset W\in\cA}
\text{mult}_{\bP(W)}(\beta\cdot D_{red}),
$$
where the multiplicity of rational divisors is defined by
linearity from the integral divisors. For $k$ as above and
$V\in\cG$ define
$$
a_{k,V}:=r(V)- \rndown{s_V(\beta^{(k)})} -1 +\delta_{V,0},
$$
where $\delta_{V,0}$ is $1$ if $V=0$ and is $0$ if $V\ne 0$.

\begin{lem}\label{computationLineBundle} $\cL^{(k)}=\cO_Y(-a_{k,0}E_0+\sum_{V\in\cG-\{0\}}
a_{k,V}E_V).$
\end{lem}
\begin{proof} First,
$K_{Y}=K_{Y/X}+\mu^*K_X$. We know $\omega_X=\cO_X(-n)$. Also,
\begin{align*}
K_{Y/X} &=\sum_{V\in\cG-\{0\}}(r(V)-1)E_V\\
\mu^*(\beta\cdot D_{red}) &=\sum_{V\in\cG-\{0\}}s_V(\beta)E_V,\ \
\ \ \beta\in\bZ^\cA,
\end{align*}
by \cite{Te} -Lemma 2.1. One can let $\beta\in\bQ^\cA$ in the last
formula by multiplying with a scalar that clears denominators.
Thus, writing $\cL^{(k)}$ in divisor form, the coefficient of
$E_V$ ($V\in\cG$) becomes
$$ \left\{
\begin{array}{ll}
r(V)-1-\rndown{s_V(\beta^{(k)})}\ \ \ \ &\text{ if }V\ne 0,\\
-n+\sum_{W\in\cA}\beta_W^{(k)}\ \ \ \ &\text{ if }V=0.
\end{array}
\right.
$$
This is equivalent to the claim.
\end{proof}

\begin{lem}\label{computationChernforE} The Chern character $ch(\cE_\alpha)$
is the image in $H^*(Y,\bQ)$ of  the formal power series
$$R_\alpha := P_{n-p-1}\cdot e^{\sum_{V\in\cG}a_{k,V}c_V}\ \ \ \in\bQ[[c_V]]_{V\in\cG}$$ under the
map induced by (\ref{eq. isom cohomology}) after $\otimes_\bZ\bQ$.
\end{lem}
\begin{proof} Follows by the multiplicativity of the Chern
character, from Lemma \ref{computationChernWedgeOmegaLogDual}, and
Lemma \ref{computationLineBundle}.
\end{proof}

\begin{thm}\label{computationSpectrumFinal} With $\alpha$ as
above, the multiplicity $n_\alpha (f)$ of $\alpha$ in
$\text{Sp}(f)$ is
\begin{equation}\label{Spformula}
(-1)^{n-p-1}\left ( R_\alpha\cdot G \right )_{n-1}
\end{equation}
where (\ref{Spformula}) is viewed as a number via identification
of the degree $n-1$ homogeneous part of $\bQ[[c_V]]_{V\in\cG}/I$
with $\bQ\cdot (-c_0)^{n-1}$.
\end{thm}
\begin{proof} It follows immediately from (\ref{lastRR}), Lemma
\ref{computationChernforE}, and Corollary \ref{cor. todd class
canonical resolutions}.
\end{proof}

\section{Examples}

The following  examples illustrate how Theorem
\ref{computationSpectrumFinal} works.

\medskip
\noindent (a) Consider the arrangement $\cA$ of three lines in
$\bC^2$ meeting at the origin. It is defined for example by the
equation $f=xy(x+y)\in\bC[x,y]$. Let $\cG=\{0, L_1, L_2, L_3\}$
where $L_i$ are the lines. For $V=L_i$, denote $c_V$ by $c_i$
($i=1,2,3$). The ideal $I\subset\bZ[[c_V]]_{V\in\cG}$ is generated
by $c_0^2$ and $c_0+c_i$ ($i=1,2,3$). We have (skipping the terms
of degree $\ge 2$)
\begin{align*}
&F  = 1-2c_0,\\
&G  = 1-c_0 ,\\
&H  = 1+2c_0+c_1+c_2+c_3 ,\\
&K_{0,0} = K_{1,0} = 1,\ K_{0,1}=1,\ K_{1,1}=2c_0+c_1+c_2+c_3,\\
&P_0 =1,\ P_1=1-(2c_0+c_1+c_2+c_3).\\
\end{align*}
Also, passing directly to the quotient $\bZ[[c_V]]_{V\in\cG}/I$,
we have
\begin{align*}
R_{1/3} &=1+c_0,\ \ \ \ \ \ \ & R_{4/3} &=1,\\
R_{2/3} &=1+2c_0,\ \ \ \ \ \ \ & R_{5/3} &=1+c_0.\\
R_{3/3} &=1+3c_0,\\
\end{align*}
Then, denoting by $(.)_1$ the coefficient of $-c_0$, we have
\begin{align*}
n_{1/3}&=-(R_{1/3}G)_1 =-(1-c_0^2)_1=0,\\
n_{2/3}&= -(R_{2/3}G)_1=-(1+c_0)_1=1,\\
n_1&=-(R_{3/3}G)_1 =-(1+2c_0)_1=2,\\
n_{4/3}&=(R_{4/3}G)_1=(1-c_0)_1=1,\\
n_{5/3}&=(R_{5/3}G)_1=(1-c_0^2)_1=0.
\end{align*}
Hence the spectrum of $f$ is $\text{Sp}(f)=t^{2/3}+2t+t^{4/3}$,
which is well-known.

\medskip
\noindent (b) Consider the central hyperplane arrangements of
degree 4 in $\bC^3$ given by
$$
f_1=(x^2-y^2)(x+z)(x+2z),$$
$$f_2=(x^2-y^2)(x^2-z^2).
$$
They are combinatorially equivalent. Here $\cA=\{A_i\subset\bC^3\
|\  i=1,\ldots , 4\}$, and $\cG=L(\cA)-\{\bC^3\}$ is given by
$$\{0,B_1,\ldots ,B_6,A_1,\ldots, A_4\},$$
where $B_j, A_i$ have codimension $2,$ resp. $1$, and $B_j\subset
A_i$ if $(i,j)$ lies in
$$
M:=\{
(1,1),(1,2),(1,3),(2,2),(2,5),(2,6),(3,1),(3,4),(3,6),(4,3),(4,4),(4,5)
\}.
$$
The ideal $I$ is generated by $c_{A_i}+\sum_{(i,j)\in
M}c_{B_j}+c_0$, $c_0c_C$, $c_{B_j}c_{B_{j'}}$ with $j\ne j'$,
$c_{B_j}c_0$, and $c_{B_j}^2+c_0^2$. Then, modulo $I$, we have

\begin{align*}
&F=9c_0^2-(c_{B_1}+\ldots c_{B_6})-3c_0+1,\\
&G=c_0^2-\frac{1}{2}(c_{B_1}+\ldots c_{B_6})-\frac{3}{2}c_0+1,\\
&H=c_0^2-c_0+1,\\
&P_0=1,\ \ P_1=-\frac{1}{2}c_0^2+c_0+2,\ \
P_2=\frac{1}{2}c_0^2+c_0+1, \\
\end{align*}
\begin{align*}
&R_{1/4}=\frac{1}{2}c_0^2+c_0+1, &R_{7/4}=-\frac{1}{2}c_0^2+2(c_{B_1}+\ldots +c_{B_6})+5c_0+2, \\
&R_{2/4}=2c_0^2+2c_0+1,  &R_{8/4}=\frac{11}{2}c_0^2+2(c_{B_1}+\ldots +c_{B_6})+7c_0+2,\\
&R_{3/4}=\frac{3}{2}c_0^2+(c_{B_1}+\ldots +c_{B_6})+3c_0+1,  &R_{9/4}=1,\\
&R_{4/4}=5c_0^2+(c_{B_1}+\ldots +c_{B_6})+4c_0+1,  &R_{10/4}=\frac{1}{2}c_0^2+c_0+1,\\
&R_{5/4}=-\frac{1}{2}c_0^2+c_0+2,&R_{11/4}=-c_0^2+(c_{B_1}+\ldots +c_{B_6})+2c+1.\\
&R_{6/4}=\frac{3}{2}c_0^2+3c_0+2,&\ \\
\end{align*}
Then Theorem \ref{computationSpectrumFinal} gives
\begin{align*}
\text{Sp}(f_1)=\text{Sp}(f_2)=t^{3/4}+3t+t^{6/4}-3t^2+t^{9/4}.
\end{align*}
We used Macaulay 2 for some of the computations. The spectrum in
this case can also be computed by \cite{St89} -Theorem 6.1 which
treats the case of homogeneous polynomials with 1-dimensional
critical locus. One can check that the outcome is the same as
ours. Remark that there is a shift by multiplication by $t$
between the definition of spectrum of \cite{St89} and that of this
article. Also, the beginning part of the spectrum, which is given
by inner jumping numbers by \cite{Bu-HS}, can be computed via a
different method, see \cite{Bu-JNHA} - Section 5, Example (b).

\end{document}